\documentclass[11pt]{amsart}
\usepackage{verbatim}
\usepackage{amssymb}
\newcommand{\bbc}{\mathbb{C}}
\newcommand{\bbz}{\mathbb{Z}}
\newcommand{\bbcd}{\mathbb{C}^{d}}
\newcommand{\bbzd}{\mathbb{Z}^d}
\newcommand{\bbcq}{\mathbb{C}_q}
\newcommand{\bbcqa}{\mathbb{C}_q^{(1)}}
\newcommand{\bbcqb}{\mathbb{C}_q^{(2)}}
\newcommand{\smn}{\sigma(m, n)}
\newcommand{\snm}{\sigma(n, m)}
\newcommand{\la}{\lambda}
\newcommand{\fmn}{f(m, n)}
\newcommand{\tp}{T^{\prime}}

\newcommand{\fnm}{f(n, m)}
\newcommand{\rad}{\mathrm{rad}}

\newcommand{\ad}{\mathrm{ad}}
\newcommand{\Der}{\mathrm{Der}}
\newcommand{\mg}{\mathfrak{g}}
\newcommand{\h}{\mathfrak{h}}
\newtheorem{thm}{Theorem}[section]
 
 \newtheorem{lem}[thm]{Lemma}
 \newtheorem{prop}[thm]{Proposition}
 \theoremstyle{definition}
 
 \theoremstyle{remark}
 \newtheorem{rem}[thm]{Remark}
 
 \numberwithin{equation}{section}

\begin{document}

\begin{abstract}
Let $\bbcq$ be the quantum torus associated with the $d \times d$ matrix $q = (q_{ij})$,
where $q_{ij}$ are roots of unity with $q_{ii} = 1$ and $q_{ij}^{-1} = q_{ji}$ for all $1 \leq i, j \leq d.$
Let $\Der(\bbcq)$ be the Lie algebra of all the derivations of $\bbcq$. In this paper
we define the Lie algebra $\Der(\bbcq) \ltimes \bbcq$ and classify its irreducible modules 
with finite dimensional weight spaces. These modules under certain 
conditions turn out to be of the form $V \otimes \bbcq$, 
where $V$ is a finite dimensional irreducible $gl_d$-module.
\end{abstract}

\title[]{The irreducible modules for the derivations of the rational quantum torus}

\author[Rao]{S. Eswara Rao}
\address{School of mathematics, Tata Institute of Fundamental Research,
Homi Bhabha Road, Mumbai 400005, India}
\email[S. Eswara Rao]{senapati@math.tifr.res.in}

\author[Batra]{Punita Batra}
\address{Department of Mathematics, Harish-Chandra Research Institute,
Chhatnag
Road, Jhunsi, Allahabad
211019 INDIA} 
\email[Punita Batra]{batra@hri.res.in}

\author[Sharma]{Sachin S. Sharma}
\address{School of mathematics, Tata Institute of Fundamental Research,
Homi Bhabha Road, Mumbai 400005, India}
\email[Sachin S. Sharma]{sachin@math.tifr.res.in}

\bigskip

\subjclass{17B65, 17B66, 17B68}

\keywords{Quantum torus, derivations, irreducible module, weight module.}

\maketitle \section {Introduction}
Let $A = \bbc[t_1^{\pm 1},\cdots,t_n^{\pm 1}]$ be the Laurent polynomial ring in $d$ 
commuting variables. Let $\Der(A)$ be the Lie algebra of diffeomorphisms of 
$d$-dimensional torus. It is well known that $\Der(A)$ is isomorphic to 
the derivations of Laurent polynomial ring $A$ in $d$-commuting variables. G.Shen \cite{SGY} and 
Rao \cite{RE1} gave several irreducible representation of $\Der(A)$. In a paper by 
Jiang and Meng \cite{JM}, it has been proved that classification of 
irreducible integrable modules of the full toroidal Lie algebra 
can be reduced to the classification of irreducible 
$\Der(A) \ltimes A$-modules. 
Rao \cite{RE} has given a classification of irreducible modules for 
$\Der(A) \ltimes A$ with finite dimensional weight spaces and associative action of $A$. 
These modules are of the form $V \otimes A$, where $V$ is a finite dimensional irreducible 
$gl_d$-module. Building on the work of \cite{JM} and \cite{RE}, Rao and Jiang \cite{RJ} gave 
the classification of the irreducible integrable modules for the full toroidal Lie algebra.

Let $\bbcq$ be a quantum torus associated with the $d \times d$ matrix $q = (q_{ij})$, where
$q_{ij}$ are roots of unity with $q_{ii} = 1$, $q_{ij}^{-1} = q_{ji}$ for all $1 \leq i, j \leq d.$
Let $\Der(\bbcq)$ be the Lie algebra of all the derivations of $\bbcq$. In \cite{ST} W.Lin and 
S.Tan defined a functor from $gl_d$-modules to $\Der(\bbcq)$-modules. They proved that 
for a finite dimensional irreducible $gl_d$-module $V$, $V \otimes \bbcq$ is a completely 
reducible $\Der(\bbcq)$-module  except finitely many cases and hence generalised
Rao's work \cite{RE1} for the quantum case. Liu and Zhao \cite{KZ} completed the study of these
modules by proving that the ``function $g(s)$'' defined in \cite{ST} can be 
taken as a constant function $1$ and made the structure of these modules completely clear. 
In this paper we study the representations of the Lie algebra $\Der(\bbcq) \ltimes \bbcq$.

We now give more details of the paper. Let $W$ be the Lie algebra of the 
outer derivations of $\bbcq$. Let $\bbcq^{(1)}$ and $\bbcq^{(2)}$ be the copies of 
$\bbcq$ contained in the Lie algebra $\Der(\bbcq) \ltimes \bbcq$ such that $[\bbcq^{(1)},\bbcq^{(2)}] = 0$,
$\bbcq^{(1)} \cap \bbcq^{(2)} = Z(\bbcq)$ and $\Der(\bbcq) \ltimes \bbcq = W \ltimes (\bbcq^{(1)} + \bbcq^{(2)})$.
The $\Der(\bbcq)$-module $V \otimes \bbcq$ 
is an irreducible $\Der(\bbcq) \ltimes \bbcq$-module with associative
action of $\bbcq^{(1)}$ and anti associative action of $\bbcq^{(1)}$(Proposition \ref{prop1}).
The action is associative on the intersection $Z(\bbcq)$.
The main goal of this paper is to prove the converse of Proposition \ref{prop1} (Theorem \ref{prop3}) . 

The paper is organised as follows. 
In section \ref{sec1} we begin with the definition and properties of the quantum torus $\bbcq$.
We define $\Der(\bbcq)$ action on $\bbcq$ and bracket operations on $\Der(\bbcq) \ltimes \bbcq$  
 (Prop.\ref{pr1} and Prop.\ref{pr2}). Section \ref{sec2} and section \ref{sec3}
are devoted to the proof of the Theorem \ref{prop3}. In section \ref{sec2} we compute
the actions of outer derivations of $\Der(\bbcq)$ and $\bbcq$ on $V^{\prime}$. In
section \ref{sec3} we derive the action of inner derivations on $V^{\prime}$ and complete 
the proof.

\section{Preliminaries}\label{sec1}
Let $q = (q_{ij})_{d \times d}$ be any $d \times d$ matrix with nonzero complex entries
satisfying $q_{ii} = 1$, $q_{ij}^{-1} = q_{ji}$, $q_{ij}$ are roots of unity for all $1 \leq i, j \leq d.$
Let us consider the non-commutative Laurent polynomial ring
$S_{[d]} = \bbc[t_{1}^{\pm 1},\cdots,t_{d}^{\pm 1}]$. Let $J_q$ be the two sided ideal of
$S_{[d]}$ generated by the elements $\{ t_i t_j = q_{ij}t_j t_i, t_i t_i^{-1} -1, t_i^{-1}
t_i - 1 \,\, \forall \, \,1 \leq i, j \leq d \}$.
Let $\bbcq = S_{[d]} / J_q$. Then $\bbcq$ is called the quantum torus associated with the matrix $q$. The 
matrix $q$ is called the quantum torus matrix.

For $n = (n_1, \cdots , n_d) \in \bbzd$, let $t^{n} = t_{1}^{n_1} \cdots t_{d}^{n_d}$.
Define $\sigma , f : \bbzd \times \bbzd \rightarrow \bbc^{\ast}$ by
$$\sigma(n, m) = \prod_{1 \leq i < j \leq d}{q_{ji}^{n_j m_i}}, \, \,f(n, m) = \sigma(n, m) \sigma(m,n)^{-1}.$$
Then one has the following results \cite{KRY}:
\begin{enumerate}
 \item $\sigma(n+m, s+r) = \sigma(n, s)\sigma(n, r)\sigma(m, s) \sigma(m, r).$
 \item $\fnm = \fmn^{-1} , \, \, f(n, n) = f(n, -n) = 1.$
 \item $ f(n+m, s+r) = f(n, s)f(n, r)f(m, s)f(m, r).$
 \item $ t^{n}t^{m} = \snm t^{n+m}, \, \, [t^n, t^m] = (\snm - \smn)t^{n+m},\\
t^{n}t^{m} = f(n,m)t^{m}t^{n}, \, \, \forall \, n, m, r, s \in \bbzd .$
\end{enumerate}
For $f$, let $\rad(f)$ denote the radical of $f$ which is defined by
$$\rad(f) = \{n\in \bbzd : \fnm = 1 \,\, \forall \, m \in \bbzd \} .$$
It is easy to see that $\rad(f)$ is a subgroup of $\bbzd$. As $\bbcq$ is
 $\bbzd$-graded, we define derivations $\partial_1, \partial_2, \cdots, \partial_d$ satisfying
$$\partial_i (t^n) = n_i t^n  \,\,\mathrm{for}\,\, n = (n_1, n_2, \cdots, n_d) \in \bbzd.$$
The inner derivations $\ad \,t^n (t^m) = (\snm - \smn)t^{n+m}$.
Note that for $n \in \rad(f)$, $\ad\, t^n = 0$. For $u = (u_1, u_2,\cdots, u_d)\in \bbcd$, 
define $D(u, r) = t^{r}\sum_{i = 1}^{d}{u_i \partial_i}$.

Let $\Der(\bbcq)$ be the space of all derivations of $\bbcq$. Let $\Der(\bbcq)_n$ denote
the set of homogeneous derivations of $\bbcq$ with degree $n$. Then we have the following lemma:

\begin{lem}[\cite{KRY},Lemma 2.48]
 \begin{enumerate}
  \item $\Der(\bbcq) = \bigoplus_{n \in \bbzd} \Der(\bbcq)_n $
  \item \begin{equation*}
 \Der(\bbcq)_n = \left\{
  \begin{array}{l l}
    \bbc \ad \, t^n  & \quad \,\mathrm{if} \,\,\, n \notin \rad (f) \\
    \oplus_{i =1}^{d}{\bbc \, t^{n} \partial_i} &  \quad  \, \mathrm{if} \,\, \,n \in \rad (f)  . \\
  \end{array} \right.
 \end{equation*}
 \end{enumerate}
\end{lem}

The space $\Der(\bbcq)$ is a Lie algebra with the following bracket operations:
\begin{enumerate}
 \item $[\ad\,t^s, \ad\,t^r] = (\sigma(s,r) - \sigma(r,s))\, \ad\, t^{s+r}, \, \, \forall\, r,s \notin \rad(f);$
 \item $[D(u,r),\ad\,t^s] = (u,s)\sigma(r,s)\, \ad\,t^{r+s} ,\, \, \forall \, r \in \rad(f), s \notin \rad(f), u\in \bbcd;$
 \item $[D(u,r),D(u^{\prime},r^{\prime})] = D(w,r+r^{\prime}), \,\, \forall \, r, r^{\prime} \in \rad(f) ,
 u,u^{\prime} \in \bbcd$ and where $w = \sigma(r,r^{\prime})((u,r^{\prime})u^{\prime}-(u^{\prime},r)u).$
\end{enumerate}

\begin{prop} \label{pr1}
 $\bbcq$ is a $\Der(\bbcq)$-module with the following action:
 \begin{enumerate}
  \item $D(u,r).t^n = (u,n)\sigma(r,n)t^{r+n}, \,\, \forall \,\, r\in \rad(f), n\in \bbzd, u\in \bbcd;$
  \item $\ad\, t^s .t^n = (\sigma(s,n)-\sigma(n,s))t^{s+n}, \,\, \forall \,\,s\notin \rad(f), n\in \bbzd .$
 \end{enumerate}
\end{prop}

Consider the space $\Der(\bbcq) \ltimes \bbcq$. We denote its element $(T,t^{n})$ by $T+t^{n}$ where
$T \in \Der(\bbcq)$ and $t^{n} \in \bbcq$ for $n \in \bbzd$.

\begin{prop}\label{pr2}
 $\Der(\bbcq)\ltimes \bbcq$ is a Lie algebra with the following brackets:
 \begin{enumerate}
 \item $[D(u,r),t^n] = (u,n)\sigma(r,n)t^{r+n}, \,\, \forall \, r\in \rad(f), n\in \bbzd, u\in \bbcd;$
 \item $[\ad\, t^s, t^n] = (\sigma(s,n)-\sigma(n,s))t^{s+n}, \,\, \forall s\notin \rad(f), n\in \bbzd;$
 \item $[t^{m},t^{n}] = (\sigma(m,n) - \sigma(n,m))t^{m+n} , \,\, \forall \,\,m,n \in \bbzd .$
 \end{enumerate}
\end{prop}

Let $\tilde{\h} = \{ D(u,0) : u\in \bbcd\} \oplus \bbc$. Then $\tilde{\h}$ is a maximal abelian subalgebra of $\mg$.
Let $W$ denote the Lie subalgebra of $\Der(\bbcq)$ generated by the elements $D(u,r)$, where
  $u \in \bbcd, r\in \rad(f)$. Let $\mg := \Der(\bbcq)\ltimes \bbcq$ and consider the Lie subalgebras 
  $\bbcqa = \bbcq$ and $\bbcqb = \mathrm{span} \{\ad\,(t^{n}) - t^{n} \mid n\in \bbz^{d}\}$. Then 
  it is easy to see that $\mg = W \ltimes (\bbcqa + \bbcqb)$. As $\bbcqb$ is isomorphic to $\bbcq$
  by a Lie algebra isomorphism $\ad\,(t^{n}) - t^{n} \rightarrow t^{n}$, we see that $\bbcqb$ has
  an associative algebra structure given by $(\ad\,(t^{n}) - t^{n})(\ad\,(t^{m}) - t^{m}) = 
  \sigma(n,m)(\ad\,(t^{n+m}) - t^{n+m})$.

Let $V$ be a finite dimensional irreducible $gl_d$-module and $\alpha \in \bbc^{d}$.  Liu and Zhao \cite{KZ}, 
also see \cite{ST}, proved that $M^\alpha(V) = V \otimes \bbcq$ is a $\Der(\bbcq)$-module with the following actions:
\begin{enumerate}
 \item $\ad\,t^{s}v(n) = (\sigma(s,n) - \sigma(n,s))v(n+s)$ ;
 \item $D(u,r)v(n) = \sigma(r,n)((u,n+\alpha) + r u^{T})v(r+n)$,
 \end{enumerate}
 where $r u^{T} = \sum_{i,j}{r_i u_jE_{ij}}$ for 
 $r = (r_1,\cdots,r_d) \in \bbzd$ and $u = (u_1,\cdots,u_d)^{T}$, 
 and $v(n) := v \otimes t^{n} \in V(n) := V \otimes t^{n}$, $n \in \bbzd, s\notin \rad(f), r\in \rad(f), u, \alpha \in \bbcd$.

Let $V$ denote an irreducible finite dimensional $gl_d$-module. Then we have the following proposition:

\begin{prop}\label{prop1}
 $V \otimes \bbcq$ is an irreducible $\Der(\bbcq) \ltimes \bbcq$-module with the following actions:
 \begin{enumerate}
 \item $\ad\,t^{s}v(n) = (\sigma(s,n) - \sigma(n,s))v(n+s);$ 
 \item $D(u,r)v(n) = \sigma(r,n)((u,n+\alpha) + r u^{T})v(r+n);$
 \item $t^m v(n) = \sigma(m,n)v(m+n),$
 where $v(n) := v \otimes t^{n} \in V(n) := V \otimes t^{n},$ 
 $ s\notin \rad(f), r\in \rad(f), u, \alpha \in \bbcd, m,n \in \bbzd$.
\end{enumerate}
\end{prop}
\begin{proof}
 It is routine check to show that $V \otimes \bbcq$ is $\mg$-module. 
To prove the irreducibility of $V^{\prime} = V \otimes \bbcq$, let $M$ be a nonzero
submodule of $V^{\prime}$. Then since $M$ is a weight module, we have
$M = \oplus_{n\in \bbzd}{M_n \otimes t^{n}}$, where $M_n = \{v \in V: v \otimes t^{n} \in M\}$.
Let for any nonzero vector $v \in M_n$, consider the $\bbcq$ action on it.
As we note that $\bbcq$ is an irreducible $\bbcq$-module with the action $t^{n}.t^{m} = \sigma(n,m)t^{n+m}$,
we have $v \otimes \bbcq \subseteq M$. So it follows that $M_n$ is independent of $n$. So let $M_n = \bar{V} \subseteq V$.
But as $D(u,r)v \otimes t^{n} \in M$, for $v \in \bar{V}$, it follows that $\bar{V}$ is a nonzero $gl_d$-submodule
of $V$. So $\bar{V} = V$ as $V$ is an irreducible $gl_d$-module and hence $M = V^{\prime}$.
\end{proof}
\begin{rem}
 The irreducibility of $\mg$-module $V \otimes \bbcq$ also follows from Proposition 4.1 of \cite{GLKZ}, by 
 considering $V \otimes \bbcq$ as $W \ltimes \bbcq^{(1)}$-module.
\end{rem}

 We will denote the $\mg$-module in Proposition \ref{prop1} by $F^{\alpha}(V)$. It is
 trivial to see that the actions of $\bbcqa$ and $\bbcqb$ on $F^{\alpha}(V)$ associative $(x(yv)= (xy)v)$ and 
 anti-associative $(x(yv)= -(yx)v)$ respectively.
Our main aim in this paper is to prove the converse of Proposition \ref{prop1} which is as follows:
\begin{thm}\label{prop3}
 Let $V^{\prime}$ be an irreducible $\bbzd$-graded $\mg$-module with finite dimensional weight 
 spaces with respect to $\tilde{\h}$, with associative $\bbcqa$ and anti-associative $\bbcqb$ action and
 $t^{0} = 1$. Then $V^{\prime} \cong F^{\alpha}(V)$ for some $\alpha \in \bbc^{d}$ and a finite dimensional irreducible
 $gl_d$-module $V$.
\end{thm}

\section{The action of $D(u,r)$ and $\bbcq$ on $V^{\prime}$}\label{sec2}
Let $U(\mg)$ denote the universal enveloping algebra of $\mg$. Let $L(\mg)$
be a two sided ideal of $U(\mg)$ generated by $\{ t^{m}t^{n} - \sigma(m,n)t^{m+n} \,\, \forall\, m,n \in \bbzd
, (\ad\,(t^{n}) - t^{n})(\ad\,(t^{m}) - t^{m}) + \sigma(m,n)(\ad\,t^{m+n} - t^{m+n})\,\,\forall\,\, n,m \in \bbz^{d},t^{0}-1 \}$.
 Throughout this section $V^{\prime}$ will be as in
Theorem \ref{prop3}. Therefore
$L(\mg)$ acts trivially on $V^{\prime}$ and 
is a $U(\mg)/L(\mg)$-module. Let $V^{\prime} = \oplus_{r\in \bbzd}{V^{\prime}_r}$ be its weight space decomposition with
$V^{\prime}_r = \{v\in V^{\prime}: D(u,0)v = (u,r+\alpha)v,\,\, \forall \, u\in \bbcd \} $ for some fixed $\alpha \in \bbc^{d}$.

\begin{prop}\label{thm1}
For $r, s \in \bbc^{d}$, we have
 $$[t^{-s}\ad\,t^{s},t^{-r}\ad\,{t^{r}}]=0 \,\,\mathrm{on} \,\,V^{\prime}.$$
\end{prop}
\begin{proof}
First using associativity and anti-associativity of respective rational quantum tories  we get
\begin{align}\label{neq}
 \ad\,t^{r}\ad\,t^{s} - (t^{r}\ad\,t^{s}+ t^{s}\ad\,t^{r}) + 
 \sigma(s,r)\,\ad\,t^{r+s} = 0 \,\, \mathrm{on} \,\, V^{\prime} .
 \end{align}
  Let us consider
 \begin{align*}
 [t^{-s}\ad\,t^{s}&,t^{-r}\ad\,t^{r}] \\
 &= [t^{-s}\ad\,t^{s},t^{-r}]\ad\,t^{r} + t^{-r}[t^{-s}\ad\,t^{s},\ad\,t^{r}]\\
 &= [t^{-s},t^{-r}]\ad\,t^{s}\ad\,t^{r} + t^{-s}[\ad\,t^{s},t^{-r}]\ad\,t^{r} \\
 & + t^{-r}[t^{-s},\ad\,t^{r}]\ad\,t^{s} + t^{-r}t^{-s}[\ad\,t^{s},\ad\,t^{r}]\\
 &= (\sigma(s,r)-\sigma(r,s))t^{-(s+r)}(t^{s}\ad\,t^{r} + t^{r}\ad\,t^{s} - \sigma(r,s)\ad\,t^{r+s})\\
 &+ t^{-s}(\sigma(s,-r) - \sigma(-r,s))t^{s-r}\ad\,t^{r}  \\
 &- t^{-r}(\sigma(r,-s) - \sigma(-s,r))t^{r-s}\ad\,t^{s}\\
 &+ \sigma(r,s)t^{-(s+r)}(\sigma(s,r)-\sigma(r,s))\ad\,t^{r+s}\\
 &= (\sigma(s,r)-\sigma(r,s))(\sigma(-(r+s),s)t^{-r}\ad\,t^{r} +\sigma(-(r+s),r)t^{-s}\ad\,t^{s}\\
 &- \sigma(r,s)t^{-(r+s)}\ad\,t^{r+s}) + (\sigma(s,-r) - \sigma(-r,s))\sigma(-s,s-r)t^{-r}\ad\,t^{r}\\
 &- (\sigma(r,-s) - \sigma(-s,r))\sigma(-r,r-s)t^{-s}\ad\,t^{s} \\
 & +\sigma(r,s)(\sigma(s,r)-\sigma(r,s))t^{-(s+r)}\ad\,t^{r+s}\\
 &=0 .
 \end{align*}
 
\end{proof}
We state the following lemma from \cite{HM} which will be used later.
\begin{lem}[\cite{HM}\label{thm3}, Prop.19.1(b)]
 Let $\mg^{\prime}$ be a Lie algebra which need not be finite dimensional.
 Let $(V_{1},\rho)$ be an irreducible finite dimensional module for $\mg^{\prime}$.
 We have a map $\rho: \mg^{\prime} \rightarrow \mbox{End}\,(V_1)$. Then $\rho(\mg^{\prime})$
 is a reductive Lie algebra with at most one dimensional center.
\end{lem}

Let $U_1 = U(\mg)/L(\mg)$ and let us define $\tp(u,r) = t^{-r}D(u,r) - \sigma(-r,r)D(u,0)$ as 
an element of $U_1$ for $r\in \rad(f), u\in \bbcd$. Let $\tp$ be the subspace generated by $\tp(u,r)$ for all
$u$ and $r$. Let $\mg^{\prime}$ be the Lie subalgebra generated by $\tp(u,r)$ and $t^{-s}\ad\,t^{s}$ for 
all $u \in \bbcd, \,\, r \in \rad(f)$ and $s\notin \rad(f)$. Let $I$ be a subalgebra of $\mg^{\prime}$ generated 
by the elements of the form $t^{-s}\ad\,t^{s}$. Then we have the following proposition:
\begin{prop}\label{pr4}
 $I$ is an abelian ideal of $\mg^{\prime}$.
\end{prop}
\begin{proof}
 It follows from lemma \ref{thm3} that $I$ is an abelian subalgebra. To prove that
 $I$ is an ideal of $\mg^{\prime}$ we need to prove that $[\tp(u,r),t^{-s}\ad\,t^{s}] \in I$.
 So consider
 \begin{align*}
 [\tp(u,r)&,t^{-s}\ad\,t^{s}] \\
 &= [\tp(u,r),t^{-s}]\ad\,t^{s} + t^{-s}[\tp(u,r), \ad\,t^{s}] \\
 &= t^{-s}[\tp(u,r), \ad\,t^{s}] \,\,(\mbox{as}\,\, [\tp(u,r),t^{-s}] =0 , \,\,\mbox{see} 
 \,\,\mbox{prop}.\,\, \ref{prop0}(5))\\
 &= t^{-s}[t^{-r}D(u,r),\ad\,t^{s}] - t^{-s}\sigma(-r,r)[D(u,0),\ad\,t^{s}]\\
 &= t^{-s}[t^{-r},\ad\,t^{s}]D(u,r) +t^{-s}t^{-r} [D(u,r),\ad\,t^{s}] \\
 &- t^{-s}\sigma(-r,r)(u,s)\ad\,t^{s}
 \end{align*}
 as $r \in \rad(f)$, we have $[t^{-r},\ad\,t^{s}] = 0$, 
 so we get
 \begin{align*}
[\tp(u,r),t^{-s}\ad\,t^{s}] &= \sigma(s,r)(u,s)\sigma(r,s)t^{-(s+r)}\ad\,t^{r+s} \\
&- \sigma(-r,r)(u,s)t^{-s}\ad\,t^{s} \in I.
\end{align*} 
 This completes the proof.
\end{proof}

\begin{prop}\label{prop0}
 \begin{enumerate}
  \item $[\tp(u,r),\tp(v,s)] = (v,r)\sigma(-s,s)\tp(u,r)\\
  -(u,s)\sigma(-r,r)\tp(v,s) + \sigma(s,r)\tp(w,r+s)$,\\
  where $w = \sigma(r,s)[(u,s)v-(v,r)u]$ and therefore $\tp$ is a Lie-subalgebra.
  \item $[D(v,0),\tp(u,r)] = 0$.
  \item Let $V^{\prime} = \oplus_{r\in \bbzd}{V^{\prime}}_r$ be its weight space decomposition. Then
  each $V^{\prime}_r$ is $\tp$ invariant.
  \item Each $V^{\prime}_r$ is $\tp$-irreducible.
  \item ${V^{\prime}}_r \cong {V^{\prime}}_s$ as $\tp$-module.
  \end{enumerate}
\end{prop}
\begin{proof}
 Proof of $(1),(2)$ and $(3)$ are same as in \cite{RE} (Proposition 3.2).
For proving $(4)$, let $U(\mg) = \oplus_{r\in \bbzd}{U_{r}}$, where $U_r = \{v\in U(\mg):[D(u,0),v] = (u,r)v,
\, \forall \, u\in \bbcd\ \,\}$. As $V^{\prime}$ is an irreducible $\mg$-module, again by using the same argument
as Proposition 3.2 of \cite{RE} we get that $V^{\prime}_r$ is
an irreducible $U_0$-module and every element of $U_0$ can be written as a linear combination of the elements
$t^{-r_1}D(u,r_1)\cdots t^{-r_k}D(u,r_k)t^{-s_1}\ad\, t^{s_1}\cdots t^{-s_n}\ad\,t^{s_n}$.
So  $U_0$ is generated by the elements of the form $\tp(u,r)$ and $t^{-s}\ad\,t^{s}$, where $r\in \rad(f)$
and $s \notin \rad(f)$. Now we use lemma \ref{thm3}, where we take $\mg^{\prime}$ as the Lie algebra generated by
$\tp(u,r)$ and $t^{-s}\ad\,t^{s}$ and $V = V^{\prime}_r$. As the elements of the form $t^{-s}\ad\,t^{s}$ forms an 
abelian ideal (Prop.\ref{pr4}) , it follows from the Lemma \ref{thm3} that the elements of the form $t^{-s}\ad\,t^{s}$ 
must lie in
the center of $\rho(\mg^{\prime})$ which is at most one dimensional. Consequently it follows that $t^{-s}\ad\,t^{s}$
acts as a scalar on $V_r^{\prime}$ and hence $V_r^{\prime}$ is an irreducible $T^{\prime}$-module.\\
Now let us prove $(5)$.   
As $t^{s-r}V^{\prime}_{r} \subseteq V^{\prime}_{s}$. But as 
$$V^{\prime}_{r} = t^{(r-s)}t^{(s-r)}V^{\prime}_{r} \subseteq t^{(r-s)}V^{\prime}_{s} \subseteq V^{\prime}_{r}.$$
We get $V^{\prime}_{r} = t^{r-s}V^{\prime}_{s}$. Define $\psi : V^{\prime}_{r} \rightarrow V^{\prime}_{s}$ by
$\psi(v) = t^{(r-s)}v $. Note that $\psi$ is injective (as it is graded) and surjective, we need to prove that it
is a $\tp$-module homomorphism, i.e., we need to show that $[t^{(s-r)},\tp(u,m)] = 0$ which follows
from a straight forward calculation.
\end{proof}

 Now by Liu and Zhao \cite{KZ}, $\rad(f) = m_1\bbz e_1 \oplus \cdots \oplus m_d \bbz e_d$ for 
  some $0 \neq m_i \in \bbz$ for $1\leq i \leq d$. Note that this result is true only if all the entries of
  the matrix $q$ are roots of unity. Consider the Laurent polynomial ring associated with
  $\rad(f)$ which is equal to $\bbc[t_1^{\pm m_1},\cdots,t_d^{\pm m_d}] = Z(\bbcq)$ = the center of $\bbcq$.
  Let $A = \bbc[s_1^{\pm1},\cdots, s_d^{\pm1}]$ be a Laurent polynomial ring, where $s_i = t_i^{m_i}$ for
  $1 \leq i \leq d$. To avoid notational confusion, we will use the notation $d(u,r)$ for the derivations of
  $\Der(A)$, where $d(u,r) = t^{r}\sum_{i = 1}^{d}{u_i \partial_i}$, $u = (u_1, u_2,\cdots, u_d)\in \bbcd$
  and $r \in \bbzd$. Let $W$ denote the Lie subalgebra of $\Der(\bbcq)$ generated by the elements $D(u,r)$, where
  $u \in \bbcd, r\in \rad(f)$.
  We have the following proposition:
  \begin{prop}
   $\Der(A)\ltimes A \cong W \ltimes Z(\bbcq)$ with a map $\phi$ defined as
   $\phi(d(u,r)+t^s ) = \sqrt{\sigma(r,r)}D(u,r)+\sqrt{\sigma(s,s)}t^s$.
  \end{prop}
 \begin{proof}
   See Lemma 2.3 of \cite{GLKZ}.
\end{proof}

  Using the above proposition we see that from \cite{RE}, $\tp/I_2^{\prime} \cong gl_d$, where $I_2^{\prime}$
  is the ideal of $\tp$ spanned by the elements $\tp(u,r,n_1,n_2)$ which are defined as follows:
  $\tp(u,r,n_1,n_2) = \tp(u,r) - \tp(u,r + n_1) - \tp(u,r+ n_2) + \tp(u,r + n_1 + n_2)$.
  Again recall from \cite{RE} that the Lie subalgebra $T$ is isomorphic to $\tp$ under the map $\phi$
  where T is the Lie subalgebra spanned by the elements $T(u,r) = t^{-r}d(u,r) - d(u,0)$. Using this isomorphism 
  $\phi$ we see that $\phi(T(u,r)) = \sigma(r,r)\tp(u,r)$, so $\phi^{-1}(\tp(u,r)) = \sigma(r,r)^{-1}T(u,r) =
  \sigma(-r,r)T(u,r)$. As $\tp/I_2^{\prime} \cong gl_d$, we see that by the same argument as in \cite{RE},
  we have $V^{\prime}_{r} \cong V \otimes t^{r}$, where $V$ is a finite dimensional irreducible representation of $gl_d$.
  So we have $V^{\prime} \cong V \otimes \bbcq$. The isomorphism from $\tp/I_2^{\prime}$ to $gl_d$ is given by the 
  map $\pi^{\prime}$ defined by $\pi^{\prime}(\tp(e_i, e_j)) = \sigma(e_j,e_i)^{-1}E_{ji} = E_{ji}$, as $\sigma(e_i,e_j) = 1$.
  
  Now let us calculate the action of $D(u,r)$ on $V^{\prime}_n = V \otimes t^{n} := V^{\prime}(n)$.
  First consider
  \begin{align*}
  \tp(u,r)v(n) &= \sigma(-r,r)T(u,r)v(n)\\
  &=\sigma(-r,r)\sum_{i,j}{u_i r_jT(e_i, e_j)v(n)}   \\
  &= \sigma(-r,r)\sum_{i,j}{u_i r_j E_{ji}v(n)}.
  \end{align*}
  So we have
  $$t^{-r}D(u,r)v(n) = \sigma(-r,r)D(u,0)v(n) + \sigma(-r,r)\sum_{i,j}{u_i r_j E_{ji}}v(n)$$
   multiply  by $ t^{r}$  we  get
 \begin{align*}
   \sigma(-r,r)D(u,r)v(n) &= \sigma(-r,r)(u,n+\alpha)\sigma(r,n)v(n+r)\\
   &+\sigma(-r,r)(\sum_{i,j}{u_i r_j E_{ji}})\sigma(r,n)v(n+r).
 \end{align*}
  So we get
  $$ D(u,r)v(n) = \sigma(r,n)[(u,n + \alpha) + r u^{T}]v(n+r).$$
  
  Now as by Proposition \ref{prop0} we have $V^{\prime}_{r} \cong V^{\prime}_{s}$ as $T^{\prime}$-module.
  We identify $V^{\prime}_{s}$ as $t^{s}(V^{\prime}_{0})$, i.e., $t^{s}(V^{\prime}_{0}) = V^{\prime}_{s}$ for all 
  $s \in \bbzd$. Now consider
  \begin{align*}
   t^{m}t^{n}v(0) &= t^{m}v(n) , \,\,\mbox{by \,\, identification},\\
    \sigma(m,n)t^{m+n}v(0) &= t^{m}v(n) ,\\
    \sigma(m,n)v(m+n) &= t^{m}v(n) .
  \end{align*}
So far we have proved the following:
\begin{prop}
 Let $V^{\prime}$ be an irreducible $\bbzd$-graded $\mg$-module with finite dimensional weight 
 spaces with respect to $\tilde{\h}$, with associative $\bbcqa$ and anti-associative $\bbcqb$ action. 
 Then $V^{\prime} \cong V \otimes \bbcq$, where $V$ is a finite dimensional irreducible $gl_d$-module. 
 The actions of $D(u,r)$ and $\bbcq$ on $V^{\prime}$
 are given by the following:
 \begin{enumerate}
  \item  $D(u,r)v(n) = \sigma(r,n)((u,n+\alpha) + r u^{T})v(r+n)$, $\mathrm{for\,\, some\,\, fixed}$ $\alpha \in \bbc^{d};$
  \item $t^{m}v(n) = \sigma(m,n)v(m+n)$ , where $u \in \bbcd, m,n \in \bbzd, r\in \rad(f)$, and $ v(n) = v \otimes t^{n}$.
 \end{enumerate}

\end{prop}

 \section{ $\ad$ action on $V^{\prime}$ and the proof of Theorem \ref{prop3}}\label{sec3}
  
  To complete the proof of Theorem \ref{prop3} we need to determine the 
  action of $\ad$ on $V^{\prime}(n)$, which will be done in this section.
  As by the Lemma \ref{thm3} $t^{-s}\ad\,t^{s}$ acts as a scalar on $V^{\prime}(n)$. 
  Let $t^{-s}\ad\,t^{s}v(n) = \lambda(s,n)v(n)$. 
  
  \begin{prop}
   $\la(s,r)= f(r,s)\la(s,0) + \sigma(-s,s)[1 - f(r,s)]$.
  \end{prop}
  \begin{proof}
   Using
   \begin{align*}
   t^{-s}\ad \,t^{s}t^{r}v(n) &= (t^{r}t^{-s}\ad \,t^{s} + [t^{-s}\ad \,t^{s},t^{r}])v(n)
 \end{align*}
    we get the desired identity for $\la(s,r)$.
  \end{proof}

  \begin{prop}
   $\ad\, t^{s}v(n) = (\sigma(s,n) - g(s)\sigma(n,s))v(n+s)$, where 
   $ g(s) = -\sigma(s,s)\la(s,0) +1 $.
  \end{prop}
  \begin{proof}
   By above proposition we have 
   \begin{align*}
   t^{-s}\ad\,t^{s}v(n) = [f(n,s)\la(s,0) + \sigma(-s,s)(1 - f(n,s))]v(n)
   \end{align*}
   Multiply by $t^{s}$ and using the relation $t^{m}t^{n} = \sigma(m,n)t^{m+n}$, we get
   \begin{align*}
   \ad\,t^{s}v(n) &= [\sigma(s,s)f(n,s)\la(s,0) + (1 - f(n,s))]\sigma(s,n)v(n+s)\\
   &= [\sigma(s,n) - g(s)\sigma(n,s)]v(n+s).
   \end{align*}
  \end{proof}

\begin{prop}
 Let $g(s) = -\sigma(s,s)\la(s,0) +1 $ for $s \notin \rad(f)$ and $g(s)=1$ otherwise. Then
 $g(s+r) = g(s)g(r) \,\forall \, s,r \in \bbzd$.
\end{prop}
\begin{proof}
 Using the identity \ref{thm1}, it is straight forward
 calculation to show that $g$ satisfies the desired conditions. 
 \end{proof}
 
Now denote the $\Der(\bbcq)$-module $V^{\prime}$  with the above action by $G^{\alpha}_{g}(V)$.
Lin and Tan [\cite{ST}, 2004] proved that $V \otimes \bbcq = V(\psi,b) \otimes \bbcq$
is a completely reducible as $Der(\bbcq)$-module unless $(\psi,b) = (\delta_k , k), 1 \leq k \leq d-1$ or
$(\psi,b) = (0,b)$ with the following actions:
\begin{enumerate}
 \item $D(u,r)v(n) = \sigma(r,n)((u,n+\alpha) + r u^{T})v(r+n);$
 \item $\ad\,t^{s}v(n) = (\sigma(s,n)g(s) - \sigma(n,s))v(n+s),$
\end{enumerate}
where $g(s+r) = g(s)g(r) \,\forall \, s,r \in \bbzd$ and $g(r)=1 \, \forall \,\, r\in \rad(f)$. We denote these 
$\Der(\bbcq)$-modules  by $F_{g}^{\alpha}(V)$. Then we have the following proposition:

\begin{prop}\label{prop8}
 $G^{\alpha}_{g}(V) \cong F_{g^{-1}}^{\alpha}(V)$ as a $\Der(\bbcq)$-module.
\end{prop}
\begin{proof}
 First we note that $g(s) \neq 0$ for all $s \in \bbzd$. Now define a map 
 $\Psi :G^{\alpha}_{g}(V) \rightarrow  F_{g^{-1}}^{\alpha}(V)$ by
 $\Psi(v(n)) = g(n)^{-1}v(n)$. Now it is easy to prove the following:
 \begin{enumerate}
  \item $\Psi(D(u,r)v(n)) = D(u,r)\Psi(v(n))$,
  \item $\Psi(\ad \, t^{s}v(n)) = \ad \, t^{s} \Psi(v(n))$.
 \end{enumerate}
 \end{proof}
 
 To prove Theorem \ref{prop3} we have to prove that $G_{g}^{\alpha}(V) \cong F_{l}^{\beta}(V)$,  where
 $\beta \in \bbcd$ and $l$ is the constant function $1$ on $\bbzd$. To prove this we invoke \cite{KZ} 
 for the following result:
 
 \begin{thm}[Theorem 3.1,\cite{KZ}]\label{thm5}
 Let $g: \bbzd \rightarrow \bbc^{*}$  be a function satisfying $g(m)g(n) = g(m+n)$ and $g(r) = 1$ for
 any $m,n \in \bbzd$, $r\in \rad(f)$. Let $V$ be a $gl_d$-module. Then there exists $\beta \in \bbcd$
 such that $F_{g}^{\alpha}(V) \cong F_{l}^{\beta}(V)$ as $\Der(\bbcq)$-module, where $l$ denotes the
 constant function which maps all the elements of $\bbzd$ to $1$.
 \end{thm}
 
 So using Theorem \ref{thm5} and Proposition \ref{prop8} we get $G_{g}^{\alpha}(V) \cong F_{l}^{\beta}(V)$ and
 this completes the proof of Theorem \ref{prop3}.
 
 \begin{rem}
  After finishing this work, we came across a paper by Liu and Zhao, Irreducible Harish-Chandra modules over
  the derivation algebras of rational quantum tori, Glasgow Mathematical Journal Trust 2013, where they
  consider a smaller Lie algebra.
 \end{rem}
$\mathrm{\bf{Acknowledgement}}$

We thank the anonymous referee for invaluable comments and suggestions without which
our paper wouldn't be the same.

\bibliography{theirreduciblemodules.bib}
\bibliographystyle{plain}

\end{document}